\documentclass[english,11pt]{amsart}
\usepackage[T1]{fontenc}
\usepackage[latin9]{inputenc}
\usepackage{babel}
\usepackage{geometry}
\usepackage{amsthm}
\usepackage{amsmath}
\usepackage{cleveref} 
\usepackage{amssymb}
 \usepackage{amsfonts}
 \usepackage{amscd}
 \usepackage[all]{xy} 
\usepackage[dvipsnames]{xcolor}
\usepackage{pdfsync}
\usepackage{url}
\usepackage[sc]{mathpazo}

\usepackage[normalem]{ulem}

\makeatletter
  \theoremstyle{remark}
  \theoremstyle{example}
  \newtheorem*{rem*}{\protect\remarkname}
\makeatother

\usepackage{babel}
  \providecommand{\remarkname}{Remark}
  \providecommand{\lemmaname}{Lemma}

\theoremstyle{plain}
\newtheorem{theorem}{\protect\theoremname}[section]  

  \theoremstyle{definition}
  \newtheorem{example}[theorem]{\protect\examplename}
  \theoremstyle{definition}
  \newtheorem{definition}[theorem]{\protect\definitionname}
  \theoremstyle{remark}
  \newtheorem{remark}[theorem]{\protect\remarkname}
  \theoremstyle{lemma}
  \newtheorem{lemma}[theorem]{\protect\lemmaname}
    \theoremstyle{prop}
  \newtheorem{prop}[theorem]{\protect\propositionname}
    \theoremstyle{conj}
  
    \theoremstyle{question}
  
    \theoremstyle{cor}
  \newtheorem{cor}[theorem]{\protect\corname}

\makeatother

\usepackage{babel}
  \providecommand{\definitionname}{Definition}
  \providecommand{\examplename}{Example}
\providecommand{\theoremname}{Theorem}
\providecommand{\propositionname}{Proposition}
\providecommand{\conjname}{Conjecture}
\providecommand{\questionname}{Question}
\providecommand{\corname}{Corollary}

\newcommand{\T}{T}  			
\newcommand{\CStar}{{\mathbb{C}^*}}
\newcommand{\CStarN}{\left({\CStar}\right)^N}
\newcommand{\CStarNplusOne}{\left({\CStar}\right)^{N+1}}

\newcommand{\X}{X}
\newcommand{\linearf}{f} 
\newcommand{\Hf}{H} 

\newcommand{\Hone}{H_P^{(1)}}
\newcommand{\Htwo}{H_P^{(2)}} 

\newcommand{\Hd}{H_P^{(d)}}
\newcommand{\Hk}{H_P^{(k)}}
\newcommand{\HkMinusOne}{H_P^{(k-1)}}

\newcommand{\ML}{ML}  
\DeclareMathOperator{\MLdeg}{MLdeg}   

\newcommand{\Sa}{S_\alpha}
\newcommand{\SaBar}{\bar S_\alpha}
\newcommand{\Sb}{S_\beta}
\newcommand{\Srho}{S_\rho}
\newcommand{\Stau}{S_\tau}
\newcommand{\eab}{e_{\alpha,\beta}}
\newcommand{\mfS}{\mathfrak{S}}

\newcommand{\V}{V}
\DeclareMathOperator{\Flag}{Flag}   
\newcommand{\cF}{\mathcal F}  

\newcommand{\bC}{\mathbb C}

\DeclareMathOperator{\reg}{reg}

\begin{document}

\title{\vspace{-4ex}  Computing Euler obstruction functions using maximum likelihood degrees}

\author[Rodriguez]{Jose Israel Rodriguez}
\address{
Jose Israel Rodriguez\\
The University of Chicago\\
Dept. of Statistics\\
5734 S. University Ave.\\
Chicago, IL 60637.
}
\email{joisro@Uchicago.edu}
\urladdr{http://home.uchicago.edu/~joisro}

\author[Wang]{Botong Wang}
\address{
Botong Wang\\
University of Wisconsin-Madison\\
 Department of Mathematics\\
Van Vleck Hall, 480 Lincoln Drive, 
Madison, WI}
\email{wang@math.wisc.edu}
\urladdr{http://www.math.wisc.edu/~wang/}

\maketitle\vspace{-4ex} 
\begin{abstract}
\noindent
We give a numerical algorithm computing Euler obstruction functions using maximum likelihood degrees. 
The maximum likelihood degree is a well-studied property of a variety in algebraic statistics and computational algebraic geometry. In this article we use this degree to give a new way to compute Euler obstruction functions. We define the maximum likelihood obstruction function and show how it coincides with the Euler obstruction function. 
With this insight, we are able to bring new tools of computational algebraic geometry to study Euler obstruction functions. 
\end{abstract}

\section{Introduction}\label{sectionintro}
The topology of algebraic varieties is a well-studied subject in algebraic geometry. 
A celebrated result is the generalization of Chern classes to singular algebraic varieties by Macpherson, which confirmed a conjecture of Deligne and Grothendieck.
 In \cite{Mac74}, Macpherson defined Chern-Schwartz-Macpherson (CSM) classes by introducing the (local) Euler obstruction function. The Euler obstruction function is a canonically defined invariant attached to any complex algebraic or analytic variety. 
 By the formula in \cite[Section 3]{Mac74}, knowing the Euler obstruction of a variety is equivalent to knowing the Chern-Schwartz-Macpherson classes of the variety. 

Given an algebraic (or analytic) variety $X$, its Euler obstruction function $Eu_X$ is an integer valued function on $X$. It is constant on each stratum of any Whitney stratification, and hence a constructible function. Even though the definition of the $Eu_X$ is global, its value at a point $P\in X$ only depends on the analytic germ of $X$ at $P$. Thus, the Euler obstruction function is essentially a local invariant. In \cite{Mac74}, Macpherson defined  $Eu_X$ as an obstruction of extending a holomorphic function on the Nash blowup. Because of its abstract nature, it is hard to compute the value of the Euler obstruction using this definition. In \cite{BLS00}, an equivalent definition of $Eu_X$ is given using Euler characteristics of some complex links (see also \cite{Dim04}). To compute  $Eu_X$ using the second definition requires knowing a Whitney stratification and computing the Euler characteristics of various complex links (see Definition \ref{defnEuobfun}). 

The purpose of this paper is to introduce an algorithm to compute the value of the Euler obstruction function $Eu_X$ at any given point $P\in X$ using the coordinates of $P$ and the defining equations of the algebraic variety $X$ as an input. 
The algorithm does not require any knowledge of a Whitney stratification. 
We use numerical computation of maximum likelihood degrees of very affine varieties to determine the values of the Euler obstruction function. 
Our algorithm computes
the number of critical points of a general monomial  on a sequence of varieties obtained from $X$ by intersecting and removing hyperplanes. 
Since the algorithm does not use the inclusion-exclusion principle, it is relatively effective and we compute examples in Section \ref{sec:examples}.
While our presentation is for very affine varieties, our methods can be generalized to affine varieties by applying a change of coordinates as seen in Example \ref{ex:umbrella}. 

\subsection{Maximum likelihood obstruction function}
Let $\T=\CStarN$ be an affine complex torus with coordinates $z_1,\dots,z_N$. 
Let $X$ be a closed pure dimensional (not necessarily irreducible) subvariety of $T$.
Consider $\eta$, the (left) invariant holomorphic $1$-form on $T$ defined as 
$$\eta:=\mu_1\frac{dz_1}{z_1}+\mu_2\frac{dz_2}{z_2}+\cdots+\mu_N\frac{dz_N}{z_N},$$
where $\mu_i\in \bC$, 
and denote its restriction to $X$ by $\eta_X$. 
When the coefficients $\mu_i$ are general, 
 $\eta_\X$ degenerates at finitely many points and this number of points remains constant. 
We define the  \textbf{maximum likelihood degree} (ML degree) of $X$, denoted by $\MLdeg(\X)$, to be the number of degeneration points of $\eta_\X$ for general $\mu_i$. 
The notion of ML degree as first introduced in \cite{CHKS06,HKS05}.
For a more geometric interpretation of this definition, we refer to \cite{Huh13}.
Moreover, in \cite{FK00} the Gaussian degree is in some cases equivalent to the ML degree.
Our convention is that the ML degree of an empty set is zero. 

\begin{remark}
By general, we mean there exists a Zariski open dense subset of $(\mu_1,\dots,\mu_N)\in\mathbb{C}^N$ where the ML degree is constant.
\end{remark}

\begin{remark}
The definition of $\MLdeg(X)$ depends on the embedding of $X$ to $T$. However, we will see later that $\MLdeg(X)$ is indeed a stratified topological invariant of $X$, and hence does not depend on the embedding. 
\end{remark}

 Let $\linearf:=\sum_{1\leq i\leq N}a_iz_i+b$ denote a linear function on $\CStarN$,
where $a_i, b\in \bC$, and at least one of the $a_i$ is nonzero.
We define a hyperplane $\Hf$ in $\T$ to be the zero set of $\linearf$. 
 Then, we have a natural closed embedding of $T\setminus \Hf$ to $\CStarNplusOne$ given by 
$$(z_1, \ldots, z_N, f): T\setminus \Hf\to \CStarNplusOne.$$ 
For a closed subvariety $X$ of $T$ and a hyperplane $\Hf\subset T$, we define $\MLdeg(X\setminus \Hf)$ to be the maximum likelihood degree of $X\setminus \Hf$ as a closed subvariety of $\CStarNplusOne$ via the above embedding. 

Let $\Hone, \ldots, \Hd$ denote general hyperplanes in $T$ passing through $P$.
\begin{definition}\label{def:Removal}
Let $X$ be a pure $d$-dimensional closed subvariety of $\T$, and let $P\in X$ be any closed point.
For  $k\in\{0,1,\dots,d+1\}$, define
 the $k$-th \textbf{removal ML degree with respect to} $P$ by  
$$r_k(P,X):=\MLdeg\left(X\cap \Hone\cap\Htwo\cap\dots\cap\HkMinusOne\setminus\Hk\right).$$ 
Our convention  is that 
$r_0(P,X)=\MLdeg(X)$ and $r_1(P,X)=\MLdeg\left(X\setminus \Hone\right)$.
When $P,X$ are clear, we simply write $r_k$.
\end{definition}
\begin{remark}\label{generaldef}
The above definition still makes sense when $P\in (\bC^*)^N$ is not contained in $X$. We will use this generalized definition in the last section to improve computational performance when we determine the value of $Eu_X$ at different points of $X$. 
\end{remark}

\begin{definition}\label{def:MLOF}
We define the \textbf{maximum likelihood obstruction function} of $X$ at $P$ by
\begin{equation}\label{eq:MLobfun}
ML_X(P):
=(-1)^d r_0+(-1)^{d-1} r_1+\cdots+r_d-r_{d+1}.
\end{equation}
\end{definition}

\begin{remark}
The removal ML degrees and $ML_X$ exhibit these properties.
\begin{itemize}
\item If $X=\{P\}$ is a point, then $ML_X(P)=-((-1)^{1}1+(-1)^{0}0)=1$. 
\item We have $r_{d+1}(P,X)$ equals the degree of $X$ minus the Hilbert-Samuel multiplicity of $X$ at $P$.

\item If $X$ is smooth and $P\in X$ is any point, we can check that $ML_X(P)=1$ as follows. The variety $X$ and all the hyperplane sections $X\cap \Hone\cap\Htwo\cap\dots\cap\Hk$ are smooth. It follows by Theorem \ref{smooth}, $(-1)^d r_0=\chi(X)$, and 
$$(-1)^{d-k}r_k=\chi\left(X\cap \Hone\cap\Htwo\cap\dots\cap\HkMinusOne\setminus \Hk\right)$$
for any $k\in\{1,\dots,d+1\}$. Since Euler characteristic is additive on locally closed subvarieties, one can derive that $\ML_X(P)=\chi(X)-\chi(X\setminus P)=1$. 

\item The definition of $\ML_X(P)$ depends on the variety $X$ and its embedding to $T$. However, similar to the ML degrees, we will see later that $\ML_X(P)$ is a stratified topological invariant on the analytic germ of $X$ at $P$. 

\item The definitions of $r_k(P, X)$ and $ML_X(P)$ still make sense even if $P\in (\bC^*)^N$ is not contained in $X$. In fact, if $P\in (\bC^*)^N\setminus X$, then $ML_X(P)=0$. 
\end{itemize}
\end{remark}


Our main result is the following. 

\begin{theorem}\label{mainthm}
Let $X$ be a pure dimensional closed subvariety of $T=\CStarN$. Then, the Euler obstruction function is equal to the maximum likelihood obstruction function, i.e.,
$$E_X(P)=\ML_X(P)$$
for any point $P\in X$. 
\end{theorem}

\begin{remark}
In the paper \cite{ST10}, the global Euler obstruction of a constructible complex is computed using some similar constructions. However, \cite{ST10} is more of a global nature. Suggested by a note of Sch\"urmann, our results can also be proved using Kashiwara's local index theorem. We will pursue this approach in another occasion. 
\end{remark}

\subsection*{Acknowledgement}
We thank Xiping Zhang for helpful conversations. We also appreciate a detailed note from J\"org Sch\"urmann explaining a different approach of our result.


\section{Euler obstruction function}
In this section we define the Euler obstruction function and emulate notation from  Dimca's book \cite{Dim04}, where more details can be found. 

\subsection{Euler obstruction of a pair}
Let $X$ be a pure dimensional complex algebraic variety, and let 
$\mfS_\Lambda(X):=\{S_\lambda\}_{\lambda\in \Lambda}$ 
be a Whitney stratification of $X$. 
Given any two strata $\Sa$ and $\Sb$, one can define an integer $\eab$, called the \emph{Euler obstruction of the pair}. Informally, this number measures the complexity of the singularity of $\Sb$ along $\Sa$. Instead of defining Euler obstruction of  a pair using characteristic cycles (see \cite[Section 2.3]{RW17} and \cite[Section 1.1]{EM99}), we will use a theorem of Kashiwara, and define the Euler obstruction of a pair using the Euler characteristic of a complex link. 

First, we define the Euler characteristic of the complex link $\V(\Sa, \Sb)$ as follows. 
\begin{itemize}
\item If $S_\alpha\not\subset \partial \Sb$, then $V(S_\alpha, \Sb)=0$, where $\partial \Sb$ denotes the boundary of $\Sb$ in $X$. 
\item If $S_\alpha\subset \partial \Sb$, we fix a point $P\in S_\alpha$, and fix an embedding of a neighborhood of $X$ at $P$ into some complex vector space. Let $B$ be a sufficiently small ball centered at $P$. 
Let $L$ be a 
subspace of dimension $\Sa+1$ in general direction, which does not contain $P$ but sufficiently close (relative to the radius of the ball$B$).
Then,
$$V(\Sa, \Sb):=\chi(\Sb\cap B\cap L).$$
\end{itemize}
The homotopy type of $\Sb\cap B\cap L$, and hence the value of $V(\Sa, \Sb)$, does not
 depend on the choice of $P$, $B$ and $L$. 
\begin{remark}
We use a slightly different notation as in \cite[Page 100]{Dim04}. Our $V(\Sa, \Sb)$
 is their $V^{\dim \Sa+1}(\Sb, \Sa)$. 
\end{remark}

We use the following result of Kashiwara to define the Euler obstructions and avoid introducing characteristic cycles. 
\begin{definition}\label{def:Kashiwara}
\cite[Theorem 2.10]{RW17}
The Euler obstruction of the pair $(\Sa, \Sb)$ of Whitney strata, denoted by $e_{\alpha, \beta}$, is defined as follows. 
\begin{itemize}
\item If $\Sa\not\subset \bar{S}_\beta$, then $e_{\alpha, \beta}=0$, where $\bar{S}_\beta$ denotes the closure of $\Sb$ in $X$. 
\item If $\Sa=\Sb$, then $e_{\alpha,\beta}=(-1)^{\dim \Sa}$.
\item If $\Sa\subset \partial \Sb$, $e_{\alpha, \beta}=(-1)^{\dim \Sa+1}V(\Sa, \Sb)$. 
\end{itemize}

\end{definition}

\subsection{Euler obstruction functions and inductive formulas}
In this section, we recall the definition of the Euler obstruction and we give a new inductive formula to compute its values.

\begin{definition}\label{defnEuobfun}
Let $X$ be a pure dimensional complex algebraic variety, and let $\mfS_\Lambda=\{S_\lambda\}_{\lambda\in \Lambda}$ be a Whitney stratification of $X$. Let $\Srho$ denote the strata containing $P$ and let $\Stau$ denote the strata $X_{reg}$.
Denote by $\Flag(P, X)$ the set of all flags of strata of $\mfS_\Lambda$,
$$\cF=(\Srho=S_{i_1}, S_{i_2}, \ldots, S_{i_{p-1}}, S_{i_p}=\Stau),$$
such that $S_{i_j}\subset \partial S_{i_{j+1}}$ for $1\leq j\leq p-1$.
We define
$$\V(\cF):=\prod_{1\leq j\leq p-1}V(S_{i_j}, S_{i_{j+1}})$$
and
$$E_X(P):=\sum_{\cF\in \Flag(P, X)}\V(\cF).$$
The function $E_X$ is called the \textbf{Euler obstruction function} of $X$. 
\end{definition}
\begin{remark}
It follows from Macpherson's original definition that the Euler obstruction function is independent of the choice of the Whitney stratification. 
\end{remark}

By the definition of $Eu_X$, we have two inductive formulas to compute its value. 
If we group up all the flags of $\Flag(P,X)$  
with the same $S_{i_2}$, we have the following inductive formula (see \cite[Proposition 4.1.37(iii)]{Dim04}):
$$E_X(P)=\sum V(\Srho, \Sa)E_X(P_\alpha)=(-1)^{\dim \Srho+1}\sum e_{\rho,\alpha}E_X(P_\alpha)$$
where $P_\alpha$ is any point in $\Sa$ and the sum is over all $\alpha\in \Lambda$ such that $\Srho\subset \partial \Sa$. 
On the other hand, if we group up all the flags of $\Flag(P,X)$ 
with the same $S_{i_{p-1}}$, we have a different inductive formula as follows,
\begin{equation}\label{inductive0}
E_X(P)=\sum E_{\SaBar}(P)\V(\Sa, \Stau)=\sum E_{\SaBar}(P)(-1)^{\dim \Sa+1}e_{\alpha, \tau},
\end{equation}
where $\overline{S}_\alpha$ is the closure of $\Sa$ in $X$ and the sum is over all $\alpha\in \Lambda\setminus \{\tau\}$ such that $\Srho\subset \SaBar$. If we let $E_{\SaBar}(P)=0$ for $P\notin \SaBar$, then we have the following proposition.
\begin{prop}\label{prop:inductive}
Under the above notations, we have
\begin{equation}\label{inductive}
E_X(P)=\sum_{\alpha\in \Lambda\setminus\{\tau\}} E_{\overline{S}_\alpha}(P)(-1)^{\dim \Sa+1}e_{\alpha, \tau}.
\end{equation}

\end{prop}
\begin{remark}\label{sufficient}
One can use induction on the dimension of $X$ to show that $E_X$ can be determined uniquely by  formula (\ref{inductive}) and that $E_X(P)=1$ if $X=\{P\}$ is a point.
\end{remark}




\section{Topological aspect of maximum likelihood degree}
Throughout this section we assume that $T=\CStarN$ is an affine torus and $X\subset T$ is a pure $d$-dimensional closed subvariety. 

\begin{theorem}\label{smooth}\cite[Theorem 1]{Huh13}
When $X$ is smooth, $\MLdeg(X)$ is equal to the signed Euler characteristic of $X$, i.e.,
$$\MLdeg(X)=(-1)^{d}\chi(X).$$
\end{theorem}
When $X$ is singular, one needs to include correction terms to the above equation. 
What these correction terms are is the content of the next theorem. 
\begin{theorem}\cite[Corollary 2.8]{RW17}
Let $\bigsqcup_{\alpha\in \Lambda}\Sa$ be a Whitney stratification of $X$, such that $\Stau=X_{\reg}$ is the smooth locus of $X$. Then, \begin{equation}\label{indexthm}
\chi(\Stau)=\sum_{\alpha\in\Lambda}e_{\alpha,\tau}\MLdeg\left(\overline{S}_\alpha\right)
\end{equation}
where $e_{\alpha,\tau}$ is the Euler obstruction of the pair $(\Sa, \Stau)$ and $\overline{S}_\alpha$ is the closure of $\Sa$ in $T$, which is a closed subvariety of pure dimension. 
\end{theorem}

Since $\overline{S}_\tau=X$ and since $e_{\tau,\tau}=(-1)^{d}$, equation (\ref{indexthm}) simplifies to 
\begin{equation}\label{MLdegEu}
(-1)^d\MLdeg(X)=\chi(\Stau)-\sum_{\alpha\in\Lambda\setminus \{\tau\}}e_{\alpha,\tau}\MLdeg\left(\overline{S}_\alpha\right).
\end{equation}

Since the Euler obstruction function does not depend on the choice of the Whitney stratification, we can choose the Whitney stratification $X=\bigsqcup_{\alpha\in \Lambda} S_\lambda$, such that $\Stau=X_{\reg}$ is the smooth locus of $X$ and $\Srho=\{P\}$ is a fixed point in $X$. 

Before proving Theorem \ref{mainthm}, we recall some standard facts about general hyperplane sections. 

\begin{prop}\label{prophyperplane}
Suppose $H\subset T$ is a general hyperplane passing through $P$. Then, the following statements hold. 
\begin{itemize}
\item If $\dim \Sa\geq 1$, then $\dim (\Sa\cap H)=\dim \Sa-1$. If $\dim \Sa=0$ and $\alpha\neq \rho$, then $\Sa\cap H=\emptyset$. 
\item The stratification $\bigsqcup_{\alpha\in \Lambda} (\Sa\cap H)$ of $X\cap H$ is a Whitney stratification.
\item If $\Sa\cap H\neq \emptyset$ and if $\alpha\neq \rho$, then the Euler obstruction of the pair $(\Sa\cap H, \Sb\cap H)$ is equal to the negative of the Euler obstruction of the pair $(\Sa, \Sb)$. In other words,
$$e_{(\Sa\cap H, \Sb\cap H)}=-e_{(\Sa, \Sb)}.$$
\end{itemize}
\end{prop}
\begin{proof}
 The first two statements follow from Bertini's Theorem. The last statement follows from Definition~\ref{def:Kashiwara}.
Indeed,  $H$ is general, when $\alpha\neq 1$, the two pairs $(\Sa\cap H, \Sb\cap H)$ and $(\Sa, \Sb)$ has the same complex link.
In particular, $V(\Sa\cap H, \Sb\cap H)=V(\Sa, \Sb)$.
 Since $\dim (\Sa\cap H)=\dim \Sa-1$, the two Euler obstructions differ by a sign. 
\end{proof}

\begin{proof}[Proof of Theorem \ref{mainthm}]
When $X=\{P\}$ is a point, we have $ML_X(P)=1$ by definition. 
Thus, by Remark \ref{sufficient}, it suffices to show that the ML obstruction function satisfies the inductive formula~\eqref{inductive} from Proposition~\ref{prop:inductive}.

First, we plug formula (\ref{MLdegEu}) into the definition of $ML_X$, i.e., equation (\ref{eq:MLobfun}). By Proposition \ref{prophyperplane} (3), we obtain the following equation:
\begin{equation}\label{sum1}
\begin{split}
ML_X(P)=&\left(\chi(\Stau)-\sum_{\alpha\in\Lambda\setminus \{\tau\}}e_{\alpha,\tau}r_0\left(P, \overline{S}_\alpha\right)\right)\\
&-\left(\chi\left(\Stau\setminus \Hone\right)-\sum_{\alpha\in\Lambda\setminus \{\tau\}}e_{\alpha,\tau}r_1\left(P, \overline{S}_\alpha\right) \right)\\
&-\left(\chi\left(\Stau\cap \Hone\setminus \Htwo\right)+\sum_{\alpha\in\Lambda\setminus \{\tau\}}e_{\alpha,\tau}r_2\left(P, \overline{S}_\alpha\right) \right)\\
&-\cdots\\
&-\left(\chi\left(\Stau\cap \Hone\cap \cdots \cap H^{(d-1)}_P\setminus H^{(d)}_P\right)
-(-1)^{(d-1)}\sum_{\alpha\in\Lambda\setminus \{\tau\}}e_{\alpha,\tau}r_d\left(P, \overline{S}_\alpha\right) \right)\\
&-\chi\left(\Stau\cap \Hone\cap \cdots \cap H^{(d)}_P\right).
\end{split}
\end{equation}
Since the Euler characteristic is additive for a stratification of algebraic variety into locally closed subvarieties, we have
$$\chi(\Stau)=\chi\left(\Stau\setminus \Hone\right)+\chi\left(\Stau\cap \Hone\setminus \Htwo\right)\cdots+\chi\left(\Stau\cap \Hone\cap \cdots \cap H^{(d)}_P\right).$$
For $\alpha\in \Lambda\setminus \{\tau\}$, we have $\dim \Sa<d$, hence the following equation:
$$ML_{\bar{S}_\alpha}(P)=(-1)^{\dim {S}_\alpha}r_0\left(P, \overline{S}_\alpha\right)+(-1)^{\dim {S}_\alpha-1}r_1\left(P, \overline{S}_\alpha\right)+\cdots+(-1)^{\dim S_\alpha-d}r_d\left(P, \overline{S}_\alpha\right).$$
Here we recall our convention that $\MLdeg(\emptyset)=0$. After rearranging the terms,
equation  (\ref{sum1}) becomes the following equality
\begin{equation}
\begin{split}
ML_X(P)&=\sum_{\alpha\in \Lambda\setminus \{\tau\}}e_{\alpha,\tau}\left(-r_0\left(P, \overline{S}_\alpha\right)+r_1\left(P, \overline{S}_\alpha\right)+\cdots+(-1)^{d+1} r_d\left(P, \overline{S}_\alpha\right)\right)\\
&=\sum_{\alpha\in \Lambda\setminus \{\tau\}}(-1)^{\dim \Sa+1}e_{\alpha, \tau}ML_{\overline{S}_\alpha}(P).
\end{split}
\end{equation}
This proves that  $ML_X$ satisfies the same inductive formula (\ref{inductive}) as $Eu_X$. 
\end{proof}

\section{Algorithms and examples}\label{sec:examples}
In  this section we provide algorithms to compute values of $Eu_X$ and illustrative examples.

\subsection{Algorithm for computing Euler obstruction functions}

The following algorithm computes the Euler obstruction function. 
\begin{itemize}
\item Input: A point $P\in(\mathbb{C}^*)^N.$ 
and a system of equations $F$ defining the $d$-dimensional variety
$X$. 
\item Output: $ML_X(P)$.
\item Procedure: 
\begin{enumerate}
\item\label{item:randomizeH} For $i=1,2\dots,d$, construct general
affine linear polynomials $\Hone,\Htwo,\dots,\Hd$ that vanish at
the point $P$.
\item\label{item:computeRemovalMLDegree} For $k=1,\dots, d$, 
%
 compute the $k$-th removal ML degree of $X$ with respect to $P$ and store this value as $r_k(P,X)$.
%
\item\label{item:computOutput} Return the value of the alternating sum in equation \eqref{eq:MLobfun}.
\end{enumerate}
\end{itemize}

The proof of correctness for the above algorithm is as follows. 

\begin{proof}[Proof of Correctness]
Item \eqref{item:randomizeH} can be constructed for the following reason.
Fix a Whitney stratification $\mfS(X)$.
Let $U$ be the set of hyerplanes defined by a regular sequence 
$\left\{\Hone,\Htwo,\dots,\Hd\right\}$ such that 
 $S$ intersects $\Hone\cap\dots\cap\Hk$ transversally for each $S\in\mfS(X)$.
Then $U$ is a dense Zariski open subset and a member can be constructed by a choosing a random linear combination of minimal generators of the ideal of $P$.
By algorithms computing ML degrees \cite{HKS05}, 
item \eqref{item:computeRemovalMLDegree} can be achieved.
Item \eqref{item:computOutput} follows from Theorem \ref{mainthm}.
\end{proof}

\subsection{Illustrative Examples}

\begin{example}

Consider a general very affine curve $\X$ of degree $D$ in 
$\left(\mathbb{C}^*\right)^2$. 
We have the following values for the removal ML degrees, where $P_0$ is a general point in $\left(\mathbb{C}^*\right)^2$ and $P_1$ is a smooth point on the curve:
$$
\begin{array}{rccc|ccc}
k: & 0 & 1 & 2 &ML_X \\
\hline
r_{k}\left(P_0,X\right): & D^2  & D^2+D &  D&	0\\
r_{k}\left(P_1,X\right): & D^2  & D^2+D &  D-1&	1.\\
\end{array}
$$
If the very affine curve $X$ is the nodal cubic we have the following values, where 
$P_2$ is the singular points of the curve:
$$
\begin{array}{rccc|ccc}
k: & 0 & 1 & 2 &ML_X \\
\hline
r_{k}\left(P_0,X\right): &7  & 10 & 3&	0\\
r_{k}\left(P_1,X\right): & 7  & 10 &  2&	1\\
r_{k}\left(P_2,X\right): & 7  & 10 &  1&  2.\\
\end{array}
$$

\end{example}


\begin{example}
In this example we consider a hyperelliptic curve $X$
defined by 
$$
(y-2)^2=(x-1)(x^5+2x+5)^2.
$$
This curve has five isolated singularities:
$$
P_1:=(\zeta_1,2),P_2:=(\zeta_2,2),P_3:=(\zeta_3,2),P_4:=(\zeta_4,2),P_5:=(\zeta_5,2),
$$
where $\zeta$ is a root of $x^5+2x+5$.
While the $x$ coordinates of the singular points cannot be written in radicals, they can be approximated numerically as 
$$\begin{array}{ccccc}
-1.20892 ,&-0.562583\pm1.23444\sqrt{-1}, & 1.16704\pm 0.940954\sqrt{-1}.
\end{array}
$$
Let $P_0$ denote a general point on the curve. 
Then we have the following removal ML degrees:
$$
\begin{array}{llllll}
r_0(P_0,X)=12&& r_1(P_0,X)=23&&r_2(P_0,X)=10\\
 r_0(P_i,X)=12&& r_1(P_i,X)=23&&r_2(P_i,X)=9 &\text{ for }i\neq 0.
\end{array}
$$
Therefore, the values of 
$ML_X$
 are 
$ML_X(P_0)=1$ and $ML_X(P_i)=2$ for $i\neq 0$.
\end{example}

\begin{example}\label{ex:umbrella}
In this example we consider a Whitney umbrella.
We apply a coordinate change to 
$x_1^2-x_2^2x_3$ taking $x_i\to x_i-1-2^{i-1}$.
We take the Whitney stratification of $X$ to be the point $(2,3,5)$, the singular points of $X$ minus the previous point, and the smooth points of $X$.
Let $P_3,P_2,P_1$ be points of the previous respective strata. 
Then, the removal ML degrees and maximum likelihood obstruction function values are as follows. 

$$
\begin{array}{rcccc|ccc}
k: & 0 & 1 & 2 &3 &ML_X \\
\hline
r_{k}\left(P_1,X\right): & 3 & 10 &  10&	2& 1\\
r_{k}\left(P_2,X\right): & 3  &10 &  10&	1&  2\\
r_{k}\left(P_3,X\right): & 3  &10 &  9 &       1&  1\\
\end{array}
$$

We see that these values agree with the values of the Euler obstruction function. 
 Indeed, we see $\Flag(P_3,X)$ has two flags: $\cF_1=(S_\rho\subset S_\tau)$ with two strata and $\cF_2$ containing all three strata. 
We have $V(\cF_1)=-1$ 
and $V(\cF_2)=2$.

\end{example}


The next example
appears in statistics and is the determinant of a Hankel matrix.

\begin{example}
Let $\X$ denote the variety defined by $1=x_1+x_2+x_3+x_4+x_5$ and 
$$0=
\det\left[\begin{array}{ccc}
12x_{1} & 3x_{2} & 2x_{3}\\
3x_{2} & 2x_{3} & 3x_{4}\\
2x_{3} & 3x_{4} & 12x_{5}
\end{array}\right].$$
Let $P_0$  a  general point of $(\mathbb{C}^*)^5\setminus X$, 
 $P_1$ denote a general smooth point of $X$, and  
$P_2$ denote a general singular point of $X$. 
Then, the removal ML degrees are as follows. 

$$
\begin{array}{rccccc|ccc}
k: & 0 & 1 & 2 & 3 & 4 &ML_X \\
\hline
r_{k}\left(P_0,X\right): & 12  & 42 &  48& 21 &  3&	0\\
r_{k}\left(P_1,X\right):  & 12 & 42 & 48 & 21 & 2 &	1\\
r_{k}\left(P_2,X\right): & 12 & 42 & 48 & 19 & 1 &	 0\\
\end{array}
$$
If we change the affine constraint to $x_1=1$, then we have the removal ML degrees.
$$
\begin{array}{rccccc|ccc}
k: & 0 & 1 & 2 & 3 & 4 & ML_X\\
\hline
r_{k}\left(P_0,X\right): & 0  & 16 &  31& 18 &  3&	0\\
r_{k}\left(P_1,X\right):  & 0  & 16 & 31 & 18 & 2&	1 \\
r_{k}\left(P_2,X\right): & 0  & 16 & 31 & 16 & 1&	0  \\
\end{array}
$$
Thus, the respective values of $ML_X$ of the two tables agree despite having different removal ML degrees. 

\end{example}



\subsection{Homotopy continuation}
In this subsection we use homotopy continuation to determine the removal ML degrees of points in $X$  from the removal ML degrees of a general point in $(\mathbb{C}^*)^N\setminus X$.

Let $X$ be a pure $d$-dimensional subvariety of $(\bC^*)^N$ with coordinates $(z_1, \ldots, z_N)$. 
Let $\gamma: [0, 1]\to (\bC^*)^N$ be a smooth path such that $\gamma(t)\in (\bC^*)^N\setminus X$ for $0\leq t<1$ and $\gamma(1)\in X$. 

For a fix a set of general directions $A$ of $k$ hyperplanes in $(\bC^*)^N$, which can be considered as an $N\times k$-matrix of rank $k$. 
For each point $\gamma(t)\in (\bC^*)^N$, there exists a column vector $\left(p_1(t), \ldots, p_k(t)\right)^T$ such that $\gamma(t)$ is contained in the affine subspace defined by 
$$A\cdot 
\left[
\begin{array}{c}
z_1\\
z_2\\
\cdots\\
z_N
\end{array}
\right]
=
\left[
\begin{array}{c}
p_1(t)\\
p_2(t)\\
\cdots\\
p_k(t)
\end{array}
\right].$$
Denote the $i$-th row of $A$ by $A_i$, and let $H^{(i)}_t$ be the hyperplane defined by 
$$A_i\cdot [z_1, z_2, \ldots, z_N]^T=p_i(t)$$
for $1\leq i\leq k$. 

\begin{prop}
With the notation above, for a general choice of $A\in \bC^{N\times k}$ and 
general~$(u_1, u_2, \ldots, u_{N+1})\in\bC^{N+1}$ 
(away from a real algebraic subset of codimension one in $\bC^{N\times k}\times \bC^{N+1}$), 
for every $t\in [0, 1]$ the hyperplanes $H^{(1)}_t, H^{(2)}_t, \ldots, H^{(k)}_t$ are in general position and the 1-form 
$$\eta_t:=u_1\frac{dz_1}{z_1}+u_2\frac{dz_2}{z_2}+\cdots+u_N\frac{dz_N}{z_N}+u_{N+1}\frac{dy}{y}$$
is a general 1-form on $X\cap H^{(1)}_t\cap H^{(2)}_t\cap \cdots \cap H^{(k-1)}_t\setminus H^{(k)}_t$ in the sense of Definition \ref{def:Removal} (and Remark \ref{generaldef}), where $y=A_k\cdot (z_1, z_2, \ldots, z_N)-p_k(t)$. 
In particular, the 1-form $\eta_t$ has only regular degeneration points on $X\cap H^{(1)}_t\cap H^{(2)}_t\cap \cdots \cap H^{(k-1)}_t\setminus H^{(k)}_t$, 
and the number of degeneration points is equal to $r_k(\gamma(t), X)$. 
\end{prop}
\begin{proof}
For a given $t$, the non-general choice of $A$ and $(u_1, u_2, \ldots, u_{N+1})$ is contained in an algebraic hypersurface in $\bC^{N\times k}\times \bC^{N+1}$.
 As $t$ varies in the interval $[0, 1]$, the union of bad locus of all $t$ is contained in a real analytic set in $\bC^{N\times k}\times \bC^{N+1}$ of real codimension one. Thus, the proposition follows. 
\end{proof}
\begin{cor}\label{cor:homotopy}
Fixing a general choice of $A$ and $(u_1, u_2, \ldots, u_{N+1})$, one  obtains all of  the degeneration points of the 1-form $\eta_1$ on $X\cap H^{(1)}_1\cap H^{(2)}_1\cap \cdots \cap H^{(k-1)}_1\setminus H^{(k)}_1$ by homotopy continuation from the degeneration points of the 1-form $\eta_0$ on $X\cap H^{(1)}_0\cap H^{(2)}_0\cap \cdots \cap H^{(k-1)}_0\setminus H^{(k)}_0$. 
\end{cor}
\begin{proof}
This follows from the above proposition and the standard coefficient-parameter theory (see \cite[Theorem 7.1.1]{SW05}). 
\end{proof}

\begin{remark}
If we need to compute the value  Euler obstruction at several points on the variety $X\subset (\bC^*)^N$, we can use the above corollary to reduce to total amount of computation as follows.
 Suppose we want to compute $r_k(Q_j, X)$ for a sequence of points $Q_j\in X$. First, make general choices of $P\in (\bC^*)^N$, $A\in \bC^{N\times k}$ and $(u_1, u_2, \ldots, u_{N+1})\in \bC^{N+1}$. 
 For each $Q_j$, let $\gamma: [0, 1]\to (\bC^*)^N$ be the line segment from $P$ to $Q_j$. 
 Since $P$ is general, $\gamma(t)\in (\bC^*)^N\setminus X$ for all $t\in [0,1)$.
  Compute the degeneration points of $\eta_0$. 
  Then, use homotopy continuation to find the degeneration points of $\eta_1$. 
\end{remark}

\begin{example}
Let $X$ denote the $3\times 3$ rank two matrices with the affine constraint $x_{11}=1$.
Let $P_0$ denote a general point in $\left(\mathbb{C}^*\right)^9\setminus X$, 
$P_1$ a general smooth point in $X$, and $P_2$ a general singular point in $X$. 
We compute the following removal ML degrees: 
$$
\begin{array}{rcccccccc|c}
k: & 0 & 1 & 2 & 3 & 4 & 5 & 6 & 7& ML_X\\
\hline
r_{k}\left(P_0,X\right): &  39 & 204 & 444 & 519 & 351 &  138 & 30 & 3 & 0\\
r_{k}\left(P_1,X\right):  &  39 & 204 & 444 & 519 & 351 &  138 & 30  &  2 & 1\\
r_{k}\left(P_2,X\right): &  39 & 204 &  444& 519 & 351 &  136 & 28 &  1& 2.
\end{array}
$$
To compute these numbers, we first compute the degeneration points  with respect to $r_k(P_0,X)$ for each $k$. 
Then, we use homotopy continuation to take these degeneration points to the degeneration points of $\eta_X$ with respect to $r_k(P_i,X)$ for $i=1,2$.
We observe a consequence of Corollary \ref{cor:homotopy} that the the numbers of the $k$-th column are bounded above by $r_k(P_0,X)$.

The values of $ML_X$ are consistent with the computation in \cite{RW17}. 
\end{example}




\bibliographystyle{abbrv}
\bibliography{mlEC_V2}
%

\end{document}